\documentclass[reqno, 12pt]{amsart}
\usepackage{amsmath}
\usepackage{amssymb}
\usepackage{amsthm}
\usepackage{mathtools}
\usepackage{tikz-cd}
\usepackage{enumerate}
\usepackage[mathscr]{eucal}
\usepackage{graphicx}

\newtheorem{theorem}{Theorem}[section]
\newtheorem{lemma}[theorem]{Lemma}
\newtheorem{proposition}[theorem]{Proposition}
\newtheorem{corollary}[theorem]{Corollary}
\newtheorem{definition}[theorem]{Definition}
\theoremstyle{definition}
\newtheorem{example}[theorem]{Example}
\newtheorem{remark}[theorem]{Remark}

\begin{document}
	\title[]{Set Convergences via bornology}
	\author{Yogesh Agarwal \and Varun Jindal}
	
	\address{Yogesh Agarwal \& Varun Jindal: Department of Mathematics, Malaviya National Institute of Technology Jaipur, Jaipur-302017, Rajasthan, India}
	\email{yagarwal247m@gmail.com, vjindal.maths@mnit.ac.in}
	
	
	\maketitle	
	\begin{abstract}
	This paper examines the equivalence between various set convergences, as studied in \cite{beer2013gap, Idealtopo, BCL}, induced by an arbitrary bornology $\mathcal{S}$ on a metric space $(X,d)$. Specifically, it focuses on the upper parts of the following set convergences: convergence deduced through uniform convergence of distance functionals on $\mathcal{S}$ ($\tau_{\mathcal{S},d}$-convergence); convergence with respect to gap functionals determined by $\mathcal{S}$ ($G_{\mathcal{S},d}$-convergence); and bornological convergence ($\mathcal{S}$-convergence). In particular, we give necessary and sufficient conditions on the structure of the bornology $\mathcal{S}$ for the coincidence of $\tau_{\mathcal{S},d}^+$-convergence with $\mathsf{G}_{\mathcal{S},d}^+$-convergence, as well as $\tau_{\mathcal{S},d}^+$-convergence with $\mathcal{S}^+$-convergence. A characterization for the equivalence of $\tau_{\mathcal{S},d}^+$-convergence and $\mathcal{S}^+$-convergence, in terms of certain convergence of nets, has also been given earlier by Beer, Naimpally, and Rodriguez-Lopez in \cite{Idealtopo}.  To facilitate our study, we first devise new characterizations for $\tau_{\mathcal{S},d}^+$-convergence and $\mathcal{S}^+$-convergence, which we call their miss-type characterizations. 
	\end{abstract}
	\section{Introduction}

	A topology on a collection of subsets of a topological space $(X,\tau)$ is called a hyperspace topology. In the literature, a number of hyperspace topologies have been studied on the collection $\mathcal{P}_0(X)$ of all nonempty subsets of a topological space $(X,\tau)$, and on $CL(X)$, the family of all nonempty closed subsets of $(X,\tau)$ \cite{ToCCoS,michael1951}. The theory of hyperspaces is fundamental to various mathematical branches such as convex analysis, functional analysis, optimization theory, and variational analysis \cite{ToCCoS, bv1996, lucchetti2006convexity, rockafellar2009variational}.
	
	For a metric space $(X,d)$, the two most studied hyperspace topologies on $CL(X)$ are the Vietoris and the Hausdorff metric topologies \cite{haus,viet-1,viet-2}. 
	However, these topologies are too strong while dealing with unbounded sets. For example, the sequence of lines $L_n = \{(x,y) : y = \frac{x}{n}\}$ appears naturally converging to the horizontal axis yet does not converge with respect to either of these two topologies. This prompted researchers to define several weaker set convergences and hyperspace topologies. In the realm of metric spaces, the interplay between set convergence and geometric set functionals has proven to be an important approach for studying various set convergences and corresponding topologies on subsets of a metric space. 
	
	The distance functional is the most useful of the set functionals. For a nonempty subset $A$ of a metric space $(X,d)$, the distance functional $d(\cdot, A)$ on $X$ corresponding to $A$ is defined as $$x \to d(x, A) = \inf\{d(x,a): a \in A\} \text{ for }x \in X.$$ So through distance functional, we can identify a closed subset $A$ of $(X,d)$ with the continuous function $d(\cdot, A) \in C(X)$. Using this identification, in 2008, Beer et al. \cite{Idealtopo} gave a unified approach to define hyperspace topologies on $\mathcal{P}_0(X)$ (or $CL(X)$). In particular, for a family $\mathcal{S}$ of nonempty subsets of a metric space $(X,d)$, they studied a hyperspace topology corresponding to which a net $(A_\lambda)$ in $CL(X)$ converges to $A \in CL(X)$ if and only if the associated net $(d(\cdot,A_{\lambda}))$ of distance functionals converges uniformly on members of $\mathcal{S}$ to $d(\cdot, A)$. This topology on $CL(X)$ is denoted by $\tau_{\mathcal{S},d}$. When $\mathcal{S}$ is $\mathcal{P}_0(X)$ and $\mathcal{B}_d(X)$ (the family of all $d$-bounded sets in $(X,d)$) respectively, the topology $\tau_{\mathcal{S},d}$ reduces to the classical Hausdorff metric topology $(\tau_{H_d})$ and Attouch-Wets topology $(\tau_{AW_d})$ respectively. Furthermore, when $\mathcal{S} = \mathcal{F}(X)$, the collection of all nonempty finite subsets of $X$, we obtain the well-known Wijsman topology $(\tau_{W_d})$. For more on these classical topologies, see (\cite{attouch1991topology,AW-1,wijsequivalence, francaviglia1985quasi,wijsconvergence}).  
	
	Two important extensions of the distance functional are the gap and excess functionals. These extensions play a central role in providing weak formulations of many hyperspace topologies (see, \cite{ToCCoS, beer1992distance, gapexcess, beer1993weak}). For a nonempty subset $S$ of a metric space $(X,d)$, the gap functional determined by $S$ is defined as $$A \to D_d(S,A) = \inf\{d(x,a): x\in S \text{ and } a \in A\} \text{ for }A \in \mathcal{P}_0(X).$$  
	Another unified approach to hyperspace topologies on subsets of a metric space $(X,d)$ is through the weak topology generated by a family of gap functionals determined by members of a family $\mathcal{S}\subseteq \mathcal{P}_0(X)$. Beer, Constantini, and Levi considered this approach formally in \cite{beer2013gap}. The weak topology determined by such gap functionals is known as the \textit{gap topology}, denoted by $\mathsf{G}_{\mathcal{S},d}$. In particular, when $\mathcal{S} = \mathcal{F}(X)$, $\mathcal{B}_d(X)$, and $\mathcal{P}_0(X)$, the gap topology reduces to Wijsman topology, bounded proximal topology, and proximal topology, respectively. 
	
	For a metric space $(X,d)$ and a family $\mathcal{S}$ of nonempty subsets of $X$, in 2004, Lechicki et al. \cite{BCL} defined a new kind of set convergence known as the bornological convergence (denoted by $\mathcal{S}$-convergence). The bornological convergence generalizes two classical set convergences: the Hausdorff metric convergence (when $\mathcal{S} = \mathcal{P}_0(X)$), and the Attouch-Wets convergence (when $\mathcal{S} = \mathcal{B}_d(X)$). The bornological convergence was further studied in \cite{beer2009operator,beer2023bornologies,Bcas,gapexcess,PbciAWc,Ucucas,Idealtopo}. 
	
	It is to be noted that for $\mathcal{S} = \mathcal{P}_0(X)$ or $\mathcal{B}_d(X)$, $\tau_{\mathcal{S},d}$-convergence and $\mathcal{S}$-convergence coincide, however in either case, the $\mathsf{G}_{\mathcal{S},d}$-convergence is weaker. As customary to hyperspace convergences, if we decompose each of these convergences into two halves, upper ($+$) and lower ($-$), then  for $\mathcal{S} = \mathcal{P}_0(X)$ or $\mathcal{B}_d(X)$, we have $\tau_{\mathcal{S},d}^{+} = \mathsf{G}_{\mathcal{S},d}^+ = \mathcal{S}^+$ but $\mathsf{G}_{\mathcal{S},d}^-$ may be strictly weaker than $\mathcal{S}^- = \tau_{\mathcal{S},d}^{-}$. 
	
	
	In this paper, we are considering the above mentioned three types of set convergences on $CL(X)$ for an arbitrary bornology $\mathcal{S}$ (see, \cite{hogbe}) on a metric space $(X,d)$. There has been tremendous interest in studying the relation among these set convergences \cite{gapexcess,PbciAWc,Idealtopo}. Beer, Naimpally, and Rodriguez-Lopez  characterized the equivalence of $\tau_{\mathcal{S},d}^{+}$-convergence and $\mathcal{S}^+$-convergence in \cite{Idealtopo} while in \cite{gapexcess} Beer and Levi characterized the equivalence of $G_{\mathcal{S},d}^{+}$-convergence and $\mathcal{S}^+$-convergence. However, the equivalence between $\tau_{\mathcal{S},d}^{+}$ and $G_{\mathcal{S},d}^{+}$ is not known. We fill this void in the present paper. We also provide a new independent characterization for the equivalence of  $\tau_{\mathcal{S},d}^{+}$ and $\mathcal{S}^+$-convergences. To study these equivalences, we first propose new characterizations for $\tau_{\mathcal{S},d}^{+}$-convergence and $\mathcal{S}^+$-convergence. While studying $\tau_{\mathcal{S},d}^+$-convergence vis-à-vis other convergences, we discover in this paper that we require a variational type of enlargement of $S \in \mathcal{S}$ in which the situation demands enlargements of different sizes at each point of $S$. We formulate this in terms of enlargement of $S \in \mathcal{S}$ by positive functions. This idea is central to our investigations. 
	
	 The paper is organized as follows. In section 2, we define various hyperspace convergences and Section 3 delves into basic convergence results. Section 4 is devoted to provide a new representation for the $\tau_{\mathcal{S},d}^+$-convergence. Some classical results are obtained as corollaries to our main theorem (Theorem \ref{Miss theorem for S topology}) of this section. The main result (Theorem \ref{gap and Stopology equivaalence}) of Section 5 unfolds the necessary and sufficient condition for the coincidence $\tau_{\mathcal{S},d}^+ = G_{\mathcal{S},d}^{+}$. This necessary and sufficient condition turns out to be a covering condition (see, Definition \ref{Strictly included sets}) on enlargements of members of $\mathcal{S}$ by positive functions.    In the final section, we present a new perspective to look at $\mathcal{S}^+$-convergence. Then the coincidence of $\tau_{\mathcal{S},d}^+$-convergence and $\mathcal{S}^+$-convergence is characterized. Several examples and counter-examples are given throughout the paper to support the findings.

	\section{Preliminaries }
	In this section, we give definitions of various set convergences and other key concepts. We start with the definition of a bornology which is central to our analysis. For a metric space $(X,d)$, a nonempty family $\mathcal{S} \subseteq \mathcal{P}_0(X)$ which is hereditary, closed under finite union, and that forms a cover of $X$ is called a \textit{bornology} on $X$. Some important bornologies on a metric space $(X,d)$ are $\mathcal{F}(X),$ $\mathcal{B}_d(X),$ and $\mathcal{P}_0(X)$. Two other useful bornologies on a metric space $(X,d)$ are $\mathcal{K}(X)$, the set of all nonempty relatively compact subsets of $X$, and $\mathcal{T}\mathcal{B}_{d}(X)$ the set of all nonempty $d$-totally bounded subsets of $X$.
	
   The open (closed) ball centered at $x\in X$ and radius $\epsilon > 0$ in $(X,d)$ is denoted by $B_d(x, \epsilon)$ ($\overline{B_d}(x, \epsilon)$). For $A \subseteq X$ and $\epsilon > 0$, the \textit{$\epsilon$-enlargement of $A$}, denoted by $B_d(A, \epsilon)$, is defined as $B_d(A, \epsilon) = \{x \in X: d(x, A) < \epsilon\}$. Note that for $\epsilon > 0, r> 0$, we have $B_d(B_d(A,\epsilon), r) \subseteq B_d(A, \epsilon +r)$.
   
   \begin{definition}
   	A metric space $(X,d)$ is called an \textit{almost convex metric space} if for $A \in \mathcal{P}_0(X)$ and $\epsilon >0, r > 0$, we have $B_d(B_d(A, \epsilon), r) = B_d(A, \epsilon + r)$. 
   \end{definition}
	Clearly, every normed linear space is an almost convex metric space.
	
	An interesting representation of the Attouch-Wets convergence in terms of enlargements of sets is as follows: a net $(A_{\lambda})$ in $\mathcal{P}_0(X)$ $\tau_{AW_d}$-converges to $A$ provided $\forall$ $\epsilon > 0$ and $B \in \mathcal{B}_{d}(X)$, $A_\lambda \cap B \subseteq B_d(A,\epsilon)$ and $A \cap B \subseteq B_d(A_\lambda, \epsilon)$ eventually.
	
	Lechicki et al. \cite{BCL} generalized this representation of Attouch-Wets convergence by replacing $\mathcal{B}_d(X)$ with any nonempty family $\mathcal{S}$ of subsets of a metric space $(X,d)$. The resulting set convergence is known as \textit{bornological convergence}. 
	
	\begin{definition}\normalfont
		Suppose $\mathcal{S}$ is a nonempty family of subsets of a metric space $(X,d)$. A net $(A_\lambda)$ is said to be \textit{lower bornological convergent} to $A$ in $\mathcal{P}_0(X)$ denoted by $\mathcal{S}^-$-convergence, if for each $\epsilon > 0$ and $S \in \mathcal{S}$, eventually $ A \cap S \subseteq B_d(A_\lambda, \epsilon)$.

		We say $(A_\lambda)$ is \textit{upper bornological convergent} to $A$ in $\mathcal{P}_0(X)$ denoted by $\mathcal{S}^+$-convergence, if for each $\epsilon > 0$ and $S \in \mathcal{S}$, eventually $ A_\lambda \cap S \subseteq B_d(A, \epsilon)$.
		
		A net $(A_\lambda)$ $\mathcal{S}$-converges to $A$ provided it is both $\mathcal{S}^-$-convergent and $\mathcal{S}^+$-convergent to $A$. The $\mathcal{S}$-convergence is known as the bornological convergence.
	\end{definition}
	
	Clearly, for $\mathcal{S} = \mathcal{B}_d(X)$ the $\mathcal{S}$-convergence is compatible with the Attouch-Wets topology $\tau_{AW_d}$. However, in general, the bornological convergence corresponding to a bornology $\mathcal{S}$ need not be topological. In the past, several authors have shown interest in exploring the conditions under which this convergence becomes topological \cite{Bcas,PbciAWc,BCL}.
	
In 2008, G. Beer et al. in \cite{Idealtopo} gave a unified approach to study hyperspace topologies through uniform convergence of distance functionals on members of a family $\mathcal{S} \subseteq \mathcal{P}_0(X)$. 
	\begin{definition}[\cite{cao}]\normalfont
		Let $(X,d)$ be a metric space and $\mathcal{S} \subseteq \mathcal{P}_0(X)$. Then for any $S\in \mathcal{S}$ and $\epsilon > 0$ sets of the form $$ [S,\epsilon]^+ = \{(A,C) \in \mathcal{P}_0(X) \times \mathcal{P}_0(X) : d(x,A) - d(x,C) < \epsilon ~~ \forall x \in S \}$$ forms a base for a quasi-uniformity on $\mathcal{P}_0(X)$. The corresponding topology is denoted by  $\tau_{\mathcal{S},d}^+$. 
	
	Similarly, $\tau_{\mathcal{S},d}^-$ is the topology on $\mathcal{P}_0(X)$ generated by the quasi-uniformity given by sets of the form 
	$$ [S,\epsilon]^- = \{(A,C) \in \mathcal{P}_0(X) \times \mathcal{P}_0(X) : d(x,C) - d(x,A) < \epsilon ~~ \forall x \in S \}.$$
	
\end{definition}	
		

%
	
	The topology $\tau_{\mathcal{S},d}$ is the supremum of $\tau_{\mathcal{S},d}^-$ and $\tau_{\mathcal{S},d}^+$. This topology is a uniformizable topology on $\mathcal{P}_0(X)$ and the family $\{U_{\mathcal{S},\epsilon} : S \in \mathcal{S}, \epsilon > 0\}$ forms a base for a compatible uniformity for $\tau_{\mathcal{S},d}$, where $$ U_{\mathcal{S},\epsilon} = \{(A,B) \in \mathcal{P}_0(X) \times \mathcal{P}_0(X) : |d(x,A) - d(x,B)| < \epsilon\ \ \forall x \in S\}.$$  
	
So a net $(A_\lambda)$ in $\mathcal{P}_0(X)$ is $\tau_{\mathcal{S},d}$-convergent to a nonempty set $A$ if and only if for any $\epsilon > 0$ and $S \in \mathcal{S}$, eventually, $|d(x,A) - d(x, A_\lambda)| < \epsilon$ for all $x \in S$, that is, if the net $(d(\cdot, A_{\lambda}))$ converges to $d(\cdot, A)$ uniformly on members of $\mathcal{S}$.
	

	 
	In \cite{beer2013gap}, the authors studied the weak topology generated by a family of gap functionals with fixed left argument from an arbitrary family $\mathcal{S}$ of subsets of a metric space $(X,d)$. This topology is known as the gap topology.  
	\begin{definition}\normalfont
		 For an arbitrary family $\mathcal{S}$ of nonempty subsets of a metric space $(X,d)$, the \textit{gap topology} is defined as the weakest topology on $\mathcal{P}_0(X)$ for which all functionals of the form $C \rightarrow D_d(S,C) \hspace{.25cm}(S \in \mathcal{S})$ are continuous. It is denoted by $\mathsf{G}_{\mathcal{S},d}$. The two halves of the gap topology $\mathsf{G}_{\mathcal{S},d}$ are:
		\begin{enumerate}
			\item [(i)] the \textit{upper gap topology} is the weakest topology on $\mathcal{P}_0(X)$ for which each member of the family of gap functionals $\{D_d(S,\cdot) : S \in \mathcal{S}\}$ is lower semi-continuous. It is denoted by $\mathsf{G}_{\mathcal{S},d}^+$.
			\item [(ii)] the \textit{lower gap topology} is the weakest topology on $\mathcal{P}_0(X)$ for which each member of the family of gap functionals $\{D_d(S,\cdot) : S \in \mathcal{S}\}$ is upper semi-continuous. It is denoted by $\mathsf{G}_{\mathcal{S},d}^-$. 
		\end{enumerate}  
		So a subbase for a compatible uniformity for the topology $\mathsf{G}_{\mathcal{S},d}$ consists of all sets of the kind:  
			$$\{(A,C) \in \mathcal{P}_0(X) \times \mathcal{P}_0(X): |D_d(S,A) - D_d(S,C)| < \epsilon\} \  \ (S \in \mathcal{S}, \epsilon > 0).$$
		
		
	\end{definition}   	
	
	\begin{remark}\label{lowergaptoplogy}
		\textit{It is to be noted that whenever the family $\mathcal{S}$ contains all singletons the lower gap topology reduces to the lower Vietoris topology} (\cite{beer2013gap}).
	\end{remark}

Let $\mathcal{S} \subseteq \mathcal{P}_0(X)$. Then for $S_1, \ldots, S_n \in \mathcal{S}$ and $\epsilon_1,\ldots,\epsilon_n > 0$, define, 
$$ \mathscr{A}_d^+(S_1, \ldots,S_n; \epsilon_1, \ldots,\epsilon_n) = \{C \in \mathcal{P}_0(X) : D_d(C, S_i) > \epsilon_i~ \text{for}~ i = 1, \ldots,n\}.$$ The collection $\{\mathscr{A}_d^+(S_1,\ldots,S_n; \epsilon_1,\ldots,\epsilon_n): n \in \mathbb{N}, S_i \in \mathcal{S}, \epsilon_i > 0, 1 \leq i \leq n\}$ together with $\mathcal{P}_0(X)$ forms a base for the upper gap topology on $\mathcal{P}_0(X)$ (\cite{beer2013gap}).
  
		In \cite{beer2013gap}, the authors gave the following important characterization for the $\mathsf{G}_{\mathcal{S},d}^+$-convergence.
	\begin{theorem}\label{Gapconvergence}
		Let $(X,d)$ be a metric space, and let $\mathcal{S} \subseteq \mathcal{P}_0(X)$. Suppose $(A_\lambda)$ is a net in $\mathcal{P}_0(X)$ and $A \in \mathcal{P}_0(X)$. Then the following assertions are equivalent:
		\begin{enumerate}[(i)]
			\item $(A_\lambda)$ $\mathsf{G}_{\mathcal{S},d}^+$-converges to $A$;
			\item for $S \in \mathcal{S}$, and $0 < \alpha < \epsilon$, whenever $A \cap B_d(S, \epsilon) = \emptyset$, then $A_\lambda \cap B_d(S, \alpha) = \emptyset$ eventually.
		\end{enumerate}
	\end{theorem}

In this paper, we give analogous characterizations for the $\tau_{\mathcal{S},d}^+$-convergence (see, Theorem \ref{Miss theorem for S topology}) and $\mathcal{S}^+$-convergence (see, Theorem \ref{upperbornologicalrepresentation}). Moreover, we study all these convergences on $CL(X)$ and when $\mathcal{S}$ is a bornology of subsets of $X$. 

	\section{Basic Relations}
	In this section, we give relations between the hyperspace convergences defined in the previous section (see, \cite{gapexcess}). We also provide some examples showing the non-equivalence of these convergences in general. 
	\begin{proposition}\label{finerthangap}
		Let $(X,d)$ be a metric space and let $\mathcal{S}$ be a bornology on $X$. Then   $\mathsf{G}_{\mathcal{S},d}^+\subseteq \tau_{\mathcal{S},d}^+$ on $CL(X)$ ($\mathcal{P}_0(X)$).	
	\end{proposition}
	\begin{proof}
		Take $A \in CL(X)$, $S_1,\ldots,S_n \in \mathcal{S}$, and $\epsilon_1,\ldots,\epsilon_n > 0 $ such that $A \in \mathscr{A}_d^+(S_1,\ldots,S_n; \epsilon_1,\ldots,\epsilon_n)$. Let $r_i = D_d(S_i,A) - \epsilon_i > 0$ for $1\leq i\leq n$. Take $r = \min\{\frac{r_i}{2}: i = 1,\ldots,n\}$ and $S = \cup_{i=1}^nS_i \in \mathcal{S}$. We claim that $[S,r]^+(A) \subseteq \mathscr{A}_d^+(S_1,\ldots,S_n; \epsilon_1,\ldots,\epsilon_n)$. To see this, consider $C \in [S,r]^+(A)$. Then $d(x, A) - d(x,C) < r \  \ \forall x \in S_i$ and for each $1\leq i\leq n$. So for each $i = 1,\ldots,n$,  $D_d(S_i,A) < r + d(x,C) \ \ \forall x \in S_i$, which implies, $D_d(S_i,A) - r \leq D_d(S_i,C)$. Thus, $C \in \mathscr{A}_d^+(S_1,\ldots,S_n;\epsilon_1,\ldots,\epsilon_n)$. 
	\end{proof}
	\begin{proposition}\label{gapfinerthanS}
		Let $(X,d)$ be a metric space and let $\mathcal{S}$ be a bornology on $X$. Then $\mathcal{S}^+ \leq \mathsf{G}_{\mathcal{S},d}^+$ on $CL(X)$, that is, every $\mathsf{G}_{\mathcal{S},d}^+$-convergent net is $\mathcal{S}^+$-convergent. 
	\end{proposition}
	\begin{proof}
		Let $(A_{\lambda})$ be a net that $\mathsf{G}_{\mathcal{S},d}^+$-converges to $A$ in $CL(X)$. Suppose by contradiction, there exist $S_1 \in \mathcal{S}$ and $\epsilon > 0$ such that $A_\lambda \cap S_1 \nsubseteq B_d(A,\epsilon)$ frequently. Then there is a cofinal set $\Lambda_1 \subseteq \Lambda$ such that $S_2 \cap B_d(A, \epsilon) = \emptyset$, where $S_2 = \{a_\lambda \in A_{\lambda} \cap S_1 : \lambda \in \Lambda_1\}$. Note that $S_2 \in \mathcal{S}$ and $A \in \mathscr{A}_d^+(S_2;\frac{\epsilon}{2})$. But $A_\lambda \cap S_2 \neq \emptyset$ for all $\lambda \in \Lambda_1$, which contradicts our assumption.      
	\end{proof}
	Using Proposition \ref{finerthangap} and Proposition \ref{gapfinerthanS}, we have the following relation between various set convergences. 
	\begin{corollary}\label{allthreerelations}
		Let $(X,d)$ be a metric space and let $\mathcal{S}$ be a bornology on $X$. Then $\tau_{\mathcal{S},d}^+$-convergence $\Rightarrow \mathsf{G}_{\mathcal{S},d}^+$-convergence $\Rightarrow \mathcal{S}^+$-convergence on $CL(X)$ $(\mathcal{P}_0(X))$. 
	\end{corollary}
	
	We now give some examples showing that the reverse implications in Corollary \ref{allthreerelations} may not hold in general. We first present an example showing that $\mathcal{S}^+ \nRightarrow \mathsf{G}_{\mathcal{S},d}^+ \nRightarrow  \tau_{\mathcal{S},d}^+$ such an example seems not to be available in the literature. 
	
	\begin{example}\label{counterxample1}
		Let $X = \{(x,y) : x\geq  0, y\geq 0\} \subseteq \mathbb{R}^2$ and $d = d_e$, the Euclidean metric. Let $\mathcal{S} = \{B \subseteq S \cup F : F \in \mathcal{F}(X)\}$, where $S = \{(0,n) : n \in \mathbb{N}\}\cup\{(0,0)\}$. Then $\mathcal{S}$ is a bornology on $X$.  Note that neither $\overline{B_d}(S, \frac{1}{2}) = \{x \in X: d(x, S) \leq \frac{1}{2
		.}\} \notin \mathcal{S}$ and nor $\overline{B_d}(S, \frac{1}{2}) = X$. So by Theorem $3$ of \cite{gapexcess}, $\mathcal{S}^+$-convergence $\nRightarrow$ $\mathsf{G}_{\mathcal{S},d}^+$-convergence on $CL(X)$.
		
		 Next we claim that $\tau_{\mathcal{S},d}^+$ is strictly finer than $\mathsf{G}_{\mathcal{S},d}^+$ on $CL(X)$. Take $A = \{(x,0) : x \geq 5 \}$. Then $A \in CL(X)$. Consider $[S, \frac{1}{2}]^+(A)$, a neighborhood of $A$ in $\tau_{\mathcal{S},d}^+$. It is enough to show that no $\mathsf{G}_{\mathcal{S},d}^+$-neighborhood of $A$ of the form $\mathscr{A}_d^+(S_1,\ldots,S_n; \epsilon_1,\ldots,\epsilon_n)$ is contained in $[S, \frac{1}{2}]^+(A)$, where $S_i \in \mathcal{S}$ and $\epsilon_i > 0$ for all $1 \leq i \leq n$.  Choose $n_0 \in \mathbb{N}$ such that $n_0> \max\{x : (x,y) \in S_i, i = 1,\ldots,n\} + 2r$, where $r = \max \{\epsilon_i : i =1,\ldots,n\}$. Let $C = \{(n_0,y) : y \geq 1\}$. Since for any $(n_0,y) \in C$ and any $(x',y') \in S_i$ for $1\leq i\leq n$, $d((n_0,y),(x',y')) \geq |n_0-x'|\geq 2r$, we have $D_d(S_i, C) \geq 2r \ \ \forall i = 1,\ldots,n$. 		
%
So $C \in \mathscr{A}_d^+(S_1,\ldots,S_n; \epsilon_1,\ldots,\epsilon_n)$. However, $C \notin [S, \frac{1}{2}]^+(A)$ as \begin{equation*}
			d((0,2n_0),A)  - d((0,2n_0), (n_0,2n_0)) = \sqrt{4n_0^2 + 25} - n_0 > n_0.
		\end{equation*} \qed	
	\end{example}
	
	\begin{example}
		Let $(X,d) = (\mathbb{R},d_u)$, where $d_u$ is the usual metric on $\mathbb{R}$ and $\mathcal{S} = \mathcal{F}(\mathbb{R})$. Then $\mathsf{G}_{\mathcal{S},d}^+ = \tau_{W_{d}}^+$. Suppose $A_n = \{\frac{1}{n}\}$ for $n \in \mathbb{N}$ and $A = \{1\}$. Then $(A_n)$ is $\mathcal{F}(\mathbb{R})^+$-convergent to $A$. However, $d(0, \{1\}) - d(0, A_n) = 1 - \frac{1}{n} > \frac{1}{2}$ for large $n$. Thus, $(A_n)$ is not $\tau_{W_{d}}^+$-convergent to $A$.\qed  
	\end{example}
	
	\section{Miss-type characterization for $\tau_{\mathcal{S},d}^+$-convergence}

The Theorem \ref{Gapconvergence} characterized $\mathsf{G}_{\mathcal{S},d}^+$-convergence in terms of enlargements of members of $\mathcal{S}$ by positive constants.  In this section, we present a parallel criterion for the $\tau_{\mathcal{S},d}^+$-convergence whenever $\mathcal{S}$ is a bornology. This new characterization, in authors opinion, makes it easier to deal with $\tau_{\mathcal{S},d}^+$-convergence. Surprisingly, in the case of $\tau_{\mathcal{S},d}^+$-convergence, enlargements of members of $\mathcal{S}$ by positive constants need not be sufficient (see, Example  \ref{counterxample1}). 

We start with some preliminary notations. Define, 
	$$ \mathbb{R}_{+}^{X} = \{f: X\to \mathbb{R}: f(x) > 0 \text{ for all } x \in X\}.$$
	For simplicity, we write $\mathcal{Z}^+ = \mathbb{R}_{+}^{X}$. Given any nonempty subset $A$ of $X$ and $f \in \mathcal{Z}^+$, the \textit{$f$-enlargement of $A$} is denoted by $B_d(A,f)$, and is defined as $$B_d(A,f) = \cup_{x \in A}B_d(x, f(x)).$$

	\begin{theorem}\label{Miss theorem for S topology}
		Let $(X,d)$ be a metric space and let $\mathcal{S}$ be a bornology on $X$. Suppose $(A_\lambda)$ is a net in $CL(X)$ and $A \in CL(X)$. Then the following statements are equivalent: 
		\begin{enumerate}[(i)]
			\item for $S \in \mathcal{S}$ and $f, g \in \mathcal{Z}^+$ with $\inf \{g(x) - f(x) : x \in S\} > 0$, whenever $A$ misses $g$-enlargement of $S$, then $(A_\lambda)$ misses the $f$-enlargement of $S$ eventually; 
			
			\item $(A_\lambda)$ is $\tau_{\mathcal{S},d}^+$-convergent to $A$. 
		\end{enumerate}
	\end{theorem}
	\begin{proof}
		$(i) \Rightarrow (ii)$. Let $S \in \mathcal{S}$ and $\epsilon > 0$. 
		If $S \subseteq B_d(A, \epsilon)$, then we are done. Otherwise, consider $S' =\{x \in S : x \notin B_d(A, \epsilon)\}$. Define $f, g \in \mathcal{Z}^+$ such that $g(x) = d(x,A)$ and $f(x) = d(x,A) - \frac{\epsilon}{2} $ for all $x \in S'$. Then $\inf_{x\in S'}(g(x)-f(x)) > 0$ and $A \cap B_d(S',g) = \emptyset$. By the hypothesis, we get a $\lambda_0$ such that $A_\lambda \cap B_d(S', f) = \emptyset$ for all $\lambda \geq \lambda_0$. This would imply that $d(x,A) - d(x, A_{\lambda}) < \epsilon$ for all $x \in S$ and $\lambda \geq \lambda_0$. Hence the net $(A_{\lambda})$ is $\tau_{\mathcal{S},d}^+$-convergent to $A$.

		$(ii) \Rightarrow (i)$. Suppose $(A_\lambda)$ is a net that $\tau_{\mathcal{S},d}^+$-converges to $A$ in $CL(X)$. If $A = X$, then there is nothing to prove. Otherwise, let $S \in \mathcal{S}$ and $ f, g \in \mathcal{Z}^+$ with $\inf \{g(x) - f(x) : x \in S\} = r > 0$ and $A \cap B_d(S,g) = \emptyset$. Then by the hypothesis, there is a $\lambda_0$ such that for all $x \in S$ and $\lambda \geq \lambda_0$, we have $d(x, A) - d(x, A_\lambda) < r$ . So whenever $\lambda \geq \lambda_0$ and $x \in S$, we have $d(x, A_\lambda) > f(x)$. Thus, $A_\lambda \cap B_d(S,f) = \emptyset$ for all $\lambda \geq \lambda_0$.  
		
	\end{proof}
 Theorem \ref{Miss theorem for S topology} can be viewed as miss-type characterization of $\tau_{\mathcal{S},d}^+$-convergence.

	It is worth mentioning that in Theorem \ref{Gapconvergence} the family $\mathcal{S}$ is not assumed to be a bornology. However, Theorem \ref{Miss theorem for S topology} need not be true for an arbitrary family $\mathcal{S} \subseteq \mathcal{P}_0(X)$ (see the next example). So it may be interesting to study the set convergence given by Theorem \ref{Miss theorem for S topology} $(i)$, specially in the context of normed linear spaces where one usually encounters families of subsets which do not form a bornology.  
%
	 \begin{example}\label{counterexanple_Tsd}
		Let $(X,d) = (\mathbb{R}^2, d_e)$.  Define $A_n = \{(x,0): x \leq 0\} \cup \{(x,y): y = 1- \frac{x}{n}, x > 0\}$ for each $n \in \mathbb{N}$, and $A = \{(x,0) : x \leq 0 \} \cup \{(x,y) : x > 0, y = 1\}$. Then $(A_n)$ is a sequence in $CL(X)$ and $A \in CL(X)$. Consider $\mathcal{S} = \{\{(x,0): x \in \mathbb{R}\}\} \cup \mathcal{F}(\mathbb{R}^2)$. Then it is not hard to verify that the sequence $(A_n)$ and $A$ satisfy the convergence given in  Theorem \ref{Miss theorem for S topology} $(i)$. However, the $\tau_{\mathcal{S},d}^+$-convergence of sequence $(A_n)$ to $A$ fails. As for $S = \{(x,0): x \in \mathbb{R}\}$ and $\epsilon = \frac{1}{2}$, we have for any $n_0 \in \mathbb{N}$, $d((n_0,0) , A) - d((n_0,0), A_{n_0}) = 1 > \epsilon$. \qed 
		\end{example}



	The following corollaries shows that for bornologies such as $\mathcal{B}_d(X)$ and $\mathcal{F}(X)$, in Theorem \ref{Miss theorem for S topology} $(i)$, it is enough to consider enlargements by positive constants rather than by members of $\mathcal{Z}^+$, that is, for $\mathcal{S} = \mathcal{B}_d(X)$ or $\mathcal{F}(X)$, we have $\tau_{\mathcal{S},d}^{+} = \mathsf{G}_{\mathcal{S},d}^{+}$ (thanks to Theorem \ref{Gapconvergence}). In the next section, we study the coincidence $\tau_{\mathcal{S},d}^{+} = \mathsf{G}_{\mathcal{S},d}^{+}$ for an arbitrary bornology $\mathcal{S}$.   
	\begin{proposition}\label{Bounded set}
		Let $(X,d)$ be a metric space and let $C$ be a bounded subset of $X$. If $f \in \mathcal{Z}^+$, then $B_d(C,f)$ is either bounded or $X$. 
	\end{proposition}
	
	\begin{proof}
		If $\sup\{f(x) : x \in C\} = r < \infty$, then $B_d(C,f) \subseteq B_d(C,r)$.  
		Otherwise, we show that $B_d(C,f) = X$. If there is a $z \in X \setminus B_d(C,f)$, then for any $x \in C$, $d(z, x) \geq f(x)$. Given $C$ is bounded, $C \subseteq B_d(x_0,r)$ for some $x_0 \in X$ and $r > 0$. Choose $n \in \mathbb{N}$ such that $d(x_0, z) < nr$.  Since the set $\{f(x): x \in C\}$ is not bounded above, we can choose $x' \in C$ such that $f(x') > (n + 2)r$. Consequently, $d(x', z) \leq d(x', x_0) + d(x_0, z) < (n+ 2)r < f(x')$. We arrive at a contradiction. 
		\end{proof}

	\begin{corollary}\label{awgap}
		Let $(X,d)$ be a metric space. Suppose $(A_\lambda)$ is a net in $CL(X)$ and $A \in CL(X)$. Then the following conditions are equivalent:
		\begin{enumerate}[(i)]
			\item for $S \in \mathcal{B}_d(X)$ and $0 < \alpha < \epsilon$, whenever $A \cap B_d(S, \epsilon) = \emptyset$, then $A_\lambda \cap B_d(S, \alpha) = \emptyset$ eventually;
			
			\item for $S \in \mathcal{B}_d(X)$ and $f, g \in \mathcal{Z}^+$ with $\inf \{g(x) - f(x) : x \in S\} > 0$, whenever $A \cap B_d(S,g) = \emptyset$, then $A_\lambda \cap B_d(S,f) = \emptyset$ eventually; 
			
			\item $(A_\lambda)$ is $\tau_{AW_d}^+$-convergent to $A$.
			
			
		\end{enumerate}
	\end{corollary}
	
	\begin{proof}$(i) \Rightarrow (ii)$.
		Let $S \in \mathcal{B}_d(X)$ and $ f, g \in \mathcal{Z}^+$ with $ \inf\{g(x) - f(x) : x \in S\} = r > 0$, satisfying, $A \cap B_d(S,g) = \emptyset$. By Proposition \ref{Bounded set}, it follows that $S' = B_d(S,f) \in \mathcal{B}_d(X)$. Since $B_d(S',r) \subseteq B_d(S, g)$, we have $A \cap B_d(S', r) = \emptyset$. Thus, by the hypothesis for $0 < \alpha < r$ there is a $\lambda_0$ such that for all $\lambda \geq \lambda_0$, we have $A_\lambda \cap B_d(S', \alpha) = \emptyset$.
		Hence $(ii)$ is established. 
		
		The implication $(ii) \Rightarrow(i)$ is immediate, and	$(ii) \Leftrightarrow (iii)$ follows from Theorem \ref{Miss theorem for S topology}.	
	\end{proof}
\begin{remark}
	It is well-known that the $\mathsf{G}_{\mathcal{B}_d(X),d}^+$ coincides with the upper bounded proximal topology (see, p.111 of \cite{ToCCoS}). So by Corollary \ref{awgap} and Theorem \ref{Gapconvergence}, the upper bounded proximal topology and upper Attouch-Wets topology coincides on $CL(X)$. 
\end{remark}
	\begin{corollary}\label{wijsman equivalence} $($Lemma $2.1.2$, \cite{ToCCoS}$)$
	Let $(X,d)$ be a metric space. Suppose $(A_\lambda)$ is a net in $CL(X)$ and $A \in CL(X)$. Then the following assertions are equivalent:
	\begin{enumerate}[(i)]
		\item $(A_\lambda)$ is $\tau_{W_{d}}^+$-convergent to $A \in CL(X)$;
		
		\item for $0 < \alpha < \epsilon$, whenever $A \cap B_d(x,\epsilon) = \emptyset$, then $A_{\lambda} \cap B_d(x,\alpha) = \emptyset$ eventually.  
	\end{enumerate}
\end{corollary}
	
	\section{Coincidence of $\tau_{\mathcal{S},d}^+$ and $\mathsf{G}_{\mathcal{S},d}^+$}
	
	We have seen earlier that $\tau_{\mathcal{S},d}^+$-convergence is stronger than the $\mathsf{G}_{\mathcal{S},d}^+$-convergence in general (see, Proposition \ref{finerthangap}). In this section, we explore the situations for the equivalence of $\tau_{\mathcal{S},d}^+$-convergence and $\mathsf{G}_{\mathcal{S},d}^+$-convergence on $CL(X)$. The following result gives a sufficient condition for the coincidence of these convergences.

	\begin{theorem}\label{suffconditiondistnceandgap}
		Let $(X,d)$ be a metric space and let $\mathcal{S}$ be a bornology on $X$. If for each $S \in \mathcal{S}$ and $f \in \mathcal{Z}^+$ either $B_d(S,f) \in \mathcal{S}$ or $\overline{B_d(S,f)} = X$, then $\tau_{\mathcal{S},d}^+ = \mathsf{G}_{\mathcal{S},d}^+$ on $CL(X)$.  
	\end{theorem}
	
	\begin{proof}
		Suppose $(A_{\lambda})$ is a net that $\mathsf{G}_{\mathcal{S},d}^+$-converges to $A$ in $CL(X)$. Pick $S \in \mathcal{S}$ and $f, g \in \mathcal{Z}^+$ with $\inf\{g(x)-f(x):x \in S\} > 0$ such that $A \cap B_d(S,g) = \emptyset$. Since $\overline{B_d(S,f)} \subseteq B_d(S,g)$, we have $B_d(S,f) \in \mathcal{S}$. Set $B_d(S,f) = S'$, and $r = \inf\{g(x) - f(x): x \in S\}$. Then $B_d(S', \frac{r}{2}) \subseteq B_d(S, g)$. Consequently, $A \cap B_d(S', \frac{r}{2}) = \emptyset$. So $A_{\lambda} \cap S' = \emptyset$ eventually as $(A_{\lambda})$ is $\mathsf{G}_{\mathcal{S},d}^+$-convergent to $A$.  Therefore, $A_{\lambda} \cap B_d(S,f) = \emptyset$ eventually. Thus, by Theorem \ref{Miss theorem for S topology}, $(A_{\lambda})$ is $\tau_{\mathcal{S},d}^+$-convergent to $A$.   	\end{proof}
	
	The converse of Theorem \ref{suffconditiondistnceandgap} need not be true in general. Consider $(\mathbb{R}, d_u)$ and take $\mathcal{S} = \mathcal{F}(X)$. Then $\tau_{\mathcal{S},d}^+ = \mathsf{G}_{\mathcal{S},d}^+ = \tau_{W_{d}}^+$ while the condition given in Theorem \ref{suffconditiondistnceandgap} fails.  
	

	In \cite{beer2013gap}, the authors introduced the notion of strictly $(\mathcal{S}-d)$ included to study the coincidence  $\mathsf{G}_{\mathcal{S},d}^+ = \mathsf{G}_{\mathcal{S},\rho}^+$ for two metrics $d, \rho$, and the coincidence $\mathsf{G}_{\mathcal{S},d}^+ = \mathsf{G}_{\mathcal{T},d}^+$  for two families $\mathcal{S}, \mathcal{T}$ of subsets of $X$. This notion is also helpful in studying the coincidence $\tau_{\mathcal{S},d}^+ = \mathsf{G}_{\mathcal{S},d}^+$.
	
	
	\begin{definition}\normalfont
		Let $(X,d)$ be a metric space and let $\mathcal{S}$ be an arbitrary family of nonempty subsets of $X$. We say a subset $A$ of $X$ is \textit{strictly $(\mathcal{S}-d)$ included} in another nonempty subset $C$ of $X$ if there exists a finite subset $\{S_1,\ldots,S_n\}$ of $\mathcal{S}$, and for every $i \in \{1,\ldots,n\}$ there are $\alpha_i, \epsilon_i$ with $0 < \alpha_i < \epsilon_i$, such that $$A \subseteq \cup_{i=1}^nB_d(S_i, \alpha_i)  \subseteq \cup_{i=1}^nB_d(S_i, \epsilon_i) \subseteq C. $$
	\end{definition}
	The above definition is a generalization of the notion of \textit{strictly $d$-included} introduced by C. Costantini et al. in \cite{costantini1993metrics} while studying when two equivalent metrics on $X$ determine the same Wijsman convergence on $CL(X)$. 
	
	\begin{proposition}\label{Strictly included sets}
		Let $(X,d)$ be a metric space and let $\mathcal{S}$ be a bornology on $X$. Suppose $A$ and $B$ are two nonempty subsets of $X$. Consider the following statements:
		
		\begin{enumerate}[(i)]
			
			\item  $A \in \mathcal{S}$ and $D_d(A,B) > 0$ ;

			\item  $A$ is strictly $(\mathcal{S}-d)$ included in $X \setminus B$. 		
			
		\end{enumerate}  
		Then $(i) \Rightarrow (ii)$ holds. 
	\end{proposition}
	
	\begin{proof}
		Let $D_d(A,B) = \epsilon$. Since $A \in \mathcal{S}$, $A \subseteq B_d(A, \frac{\epsilon}{3}) \subseteq B_d(A, \frac{\epsilon}{2}) \subseteq X \setminus B$.  
	\end{proof}

	To prove our main result of this section, we first fix some notations and prove a preliminary lemma.
	Let $\mathcal{S}$ be bornology on a metric space $(X,d)$. Denote by $\mathcal{F}(\mathcal{S})$ the family of all nonempty finite subsets of $\mathcal{S}$, and $\mathcal{F}(\mathbb{R}^+)$ for the family of all nonempty finite subsets of $\mathbb{R}^+ = (0,\infty)$.
	For any $F\in \mathcal{F}(\mathcal{S})$ choose $\alpha_F = \{\alpha_S: S\in F\} \in \mathcal{F}(\mathbb{R}^+)$. Let $(F, \alpha_F) = \{(S, \alpha_S): S\in F\}$. Consider the collection $$\mathcal{N} = \{(F, \alpha_F) : (F, \alpha_F) \in \mathcal{F}(\mathcal{S}) \times \mathcal{F}(\mathbb{R}^+)\}. $$ 
	  
	Define a relation $\leqslant$ on $\mathcal{N}$ by $(F, \alpha_F) \leqslant (F', \alpha_{F'})$ if and only if $\cup_{S\in F}B_d(S, \alpha_S) \subseteq \cup_{S' \in F'}B_d(S', \alpha_{S'})$.
	
	Set $F'' = F \cup F'$ and $\alpha_{F''} = \alpha_F\cup\alpha_{F'} = \{\alpha_{S''}: S'' \in F''\}$. Then $(F'', \alpha_{F''}) \in \mathcal{N}$. Moreover, $(F, \alpha_F) \leqslant (F'', \alpha_{F''})$, and $(F', \alpha_{F'}) \leqslant (F'', \alpha_{F''})$. Observe that the relation $\leqslant$ is reflexive and transitive on $\mathcal{N}$. Thus, $(\mathcal{N}, \leqslant)$ is a directed set.

\begin{lemma}\label{Directed set}
	Let $C$ be a nonempty open subset of a metric space $(X,d)$ and let $\mathcal{S}$ be a bornology on $X$. Then the collection $\Omega \subseteq \mathcal{N}$ defined as, 
	$$ \Omega = \left\{(F, \alpha_F) : \exists~ \epsilon_{F} \in \mathcal{F}(\mathbb{R}^+)~ \text{with}~ \alpha_S < \epsilon_S \text{ for } S \in F \text{ and }  \bigcup_{S \in F}B_d(S, \epsilon_S) \subseteq C \right\},$$
	is a directed set under the relation $\leqslant$ defined above. 
	
\end{lemma}

\begin{proof}Note that $\Omega$ is nonempty as $C$ is nonempty open and $\mathcal{S}$ is a bornology.
	Since $\Omega \subseteq \mathcal{N}$, the relation $\leqslant$ is reflexive as well as transitive on $\Omega$. It  remains to show that for any $(F, \alpha_F), (F', \alpha_{F'}) \in \Omega$ there exists $(F'', \alpha_{F''}) \in \Omega$ satisfying $(F, \alpha_F) \leqslant (F'', \alpha_{F''})$, and $(F', \alpha_{F'}) \leqslant (F'', \alpha_{F''})$. Take $(F'', \alpha_{F''})$ as above. By the above argument, we just need to show that $(F'', \alpha_{F''}) \in \Omega$.  
	
Choose $\epsilon_F$, $\epsilon_{F'} \in \mathcal{F}(\mathbb{R}^+)$ satisfying the definition of $\Omega$ corresponding to $(F,\alpha_F)$ and $(F',\alpha_{F'})$ respectively. Let $\epsilon_{F''} = \epsilon_F \cup \epsilon_{F'}$. Then, $ \alpha_{S''} < \epsilon_{S''}$ for $S''\in F''$, and $\cup_{S'' \in F''}B_d(S'', \epsilon_{S''}) \subseteq C$. Therefore, $(\Omega, \leq)$ is a directed set.\end{proof}

%
	
			\begin{theorem}\label{gap and Stopology equivaalence}
			Let $(X,d)$ be a metric space and let $\mathcal{S}$ be a bornology on $X$. Then the following statements are equivalent:
			\begin{enumerate}[(i)]
				\item $\tau_{\mathcal{S},d}^+ = \mathsf{G}_{\mathcal{S}, d}^+$ on $CL(X)$;
				
				\item for each $S \in \mathcal{S}$ and $f, g \in \mathcal{Z}^+$ with $\inf\{g(x) - f(x) : x \in S\} > 0$, either $B_d(S , f)$ is strictly $(\mathcal{S}-d)$ included in $B_d(S, g)$ or $B_d(S , g) = X$.  
			\end{enumerate}
		\end{theorem}
		\begin{proof}
			$(i) \Rightarrow (ii)$. 
			Suppose $(ii)$ fails. Choose an $S_0 \in \mathcal{S}$ and $f, g \in \mathcal{Z}^+$ with $\inf\{g(x) - f(x) : x \in S_0\} > 0$ such that neither $B_d(S_0,g) = X$ nor $B_d(S_0,f)$ is strictly $(\mathcal{S}-d)$ included in $B_d(S_0 , g)$. Consider
			\begin{equation*}
				\Omega = \left\{(F, \alpha_F) \in \mathcal{F}(\mathcal{S}) \times \mathcal{F}(\mathbb{R}^+) :\begin{split}\exists~ \epsilon_{F} \in \mathcal{F}(\mathbb{R}^+)~ \text{with}~ \alpha_S < \epsilon_S ~\text{for}~ S \in F\\ \text{and}  \bigcup_{S \in F}B_d(S, \epsilon_S) \subseteq B_d(S_0, g)\end{split}\right\}.
			\end{equation*}

			Clearly, $\Omega$ is nonempty as $(\{x\}, \{f(x)\}) \in \Omega$ for any $x \in S_0$. By Lemma \ref{Directed set}, $\Omega$ is a directed set under the relation $\leqslant$. Since $B_d(S_0, f)$ is not strictly $(\mathcal{S}-d)$ included in $B_d(S_0,g)$, for each $(F, \alpha_F) \in \Omega$ we have $x_{(F, \alpha_F)} \in B_d(S_0, f)\setminus (\cup_{S \in F}B_d(S, \alpha_S))$. Put $A = X \setminus B_d(S_0, g)$. Then $A \in CL(X)$. Now for any $(F, \alpha_F) \in \Omega$ define $A_{(F, \alpha_F)} = A \cup \{x_{(F, \alpha_F)}\}$. Consequently, $\left(A_{(F, \alpha_F)}\right)$ is a net in $CL(X)$.  We claim that the net $(A_{(F, \alpha_F)})$ $\mathsf{G}_{\mathcal{S},d}^+$-converges to $A$ while its $\tau_{\mathcal{S},d}^+$-convergence fails.  Take any finite subset $F' = \{S'_1,\ldots,S'_m\}$ of $\mathcal{S}$ and $\alpha_{F'} = (\alpha'_1,\ldots,\alpha'_m) \in \mathcal{F}(\mathbb{R}^+)$ such that $A \in \mathscr{A}_d^+(S'_1,\ldots,S'_m;\alpha'_1,\ldots,\alpha'_m)$. So $D_d(S'_i, A) > \alpha'_i$ for each $i = 1,\ldots,m$. Let $r'_i = \frac{D_d(S'_i,A) + \alpha'_i}{2}$ for $i = 1,\ldots,m$. Then $D_d(S'_i,A) > r'_i$ for $i = 1,\ldots,m$, which gives $\cup_{i = 1}^mB_d(S'_i, r'_i) \subseteq B_d(S_0, g)$. Thus, $(F', \alpha_{F'}) \in \Omega$. So there exists $\epsilon_{F'} = (\epsilon'_1,\ldots,\epsilon'_m) \in \mathcal{F}(\mathbb{R}^+)$ such that $\alpha'_i < \epsilon'_i$ for $1 \leq i \leq m $ and $(F', \epsilon_{F'}) \in \Omega$. Then for any $(F, \epsilon_F) \in \Omega$ such that $(F', \epsilon_{F'}) \leqslant (F, \epsilon_F)$ we have $\{x_{(F, \epsilon_F)}\} \in \mathscr{A}_d^+(S'_1,\ldots,S'_m;\alpha'_1,\ldots,\alpha'_m)$. Consequently, $A_{(F, \epsilon_F)}  \in \mathscr{A}_d^+(S'_1,\ldots,S'_m;\alpha'_1,\ldots,\alpha'_m)$ whenever $(F', \epsilon_{F'}) \leqslant (F, \epsilon_F)$. So the net $(A_{(F, \alpha_F)})_{(F, \alpha_F) \in \Omega}$ $\mathsf{G}_{\mathcal{S},d}^+$-converges to $A$. Now $A \cap B_d(S_0,g)
			= \emptyset$ but for every $(F, \alpha_F) \in \Omega$, we have $A_{(F, \alpha_F)} \cap B_d(S_0, f) \neq \emptyset$. Thus, by Theorem \ref{Miss theorem for S topology}, the net $\left(A_{(F, \alpha_F)}\right)_{(F, \alpha_F) \in \Omega}$ does not converge to $A$ with respect to $\tau_{\mathcal{S},d}^+$.

			$(ii)\Rightarrow (i)$. Suppose $(A_\lambda)$ is a  net in $CL(X)$ that $\mathsf{G}_{\mathcal{S},d}^+$-converges to $A \in CL(X)$. If $A = X$, then $(A_{\lambda})$ is $\tau_{\mathcal{S},d}^+$-convergent to $A$. Otherwise, consider $S \in \mathcal{S}$ and $f, g \in \mathcal{Z}^+$ with $\inf\{g(x) - f(x): x \in S\} > 0$ and $A \cap B_d(S,g) =  \emptyset$. Since $B_d(S, f)$ is strictly $(\mathcal{S}-d)$ included in $B_d(S, g)$, there exist $S_1, \ldots,S_n \in \mathcal{S}$, and $ 0 < \alpha_i < \epsilon_i$ for $i = 1,\ldots,n$ such that $B_d(S , f) \subseteq \cup_{i=1}^nB_d(S_i, \alpha_i) \subseteq \cup_{i=1}^nB_d(S_i, \epsilon_i) \subseteq B_d(S, g)$. So $A \cap (\cup_{i=1}^nB_d(S_i, \epsilon_i)) = \emptyset$. Then by Theorem \ref{Gapconvergence}, $A_\lambda \cap (\cup_{i=1}^nB_d(S_i, \alpha_i)) = \emptyset$ eventually. So  $A_\lambda \cap B_d(S , f) = \emptyset$ eventually. Consequently, by Theorem \ref{Miss theorem for S topology},  $(A_\lambda)$ $\tau_{\mathcal{S},d}^+$-converges to $A$.
			\end{proof}

		Recall that for $\mathcal{S} = \mathcal{B}_d(X)$ and $\mathcal{P}_0(X)$, the topology $\tau_{\mathcal{S},d}^+$ reduces to the upper Attouch-Wets topology, and the upper Hausdorff metric topology, respectively. So we have the following corollaries originally observed by Beer \cite{ToCCoS}.

	\begin{corollary}
		Let $(X,d)$ be a metric space. Then the upper Attouch-Wets topology on $CL(X)$ is the weakest topology such that each member of the family of set functionals $\{D_d(S,.\cdot):S \in \mathcal{B}_d(X)\}$ is lower semi-continuous. 
		
	\end{corollary}
	
	\begin{proof}
		Let $S \in B_d(X)$ and $f, g \in \mathcal{Z}^+$ with $\inf_{x \in S}(g(x) -f(x)) > 0$. Then by Proposition \ref{Bounded set}, either $B_d(S,g) \in \mathcal{B}_d(X)$ or $B_d(S,g) = X$.  If $B_d(S,g) \neq X$, then $B_d(S,f) \in \mathcal{B}_d(X)$. Further, by the choice of $f$ and $g$, $B_d(S,f)$ is far from $X \setminus B_d(S,g)$. Consequently, by Proposition \ref{Strictly included sets}, $B_d(S,f)$ is strictly $(\mathcal{S}-d)$ included in $B_d(S,g)$. Thus, the result follows from the Theorem \ref{gap and Stopology equivaalence}.       
		\end{proof}

		\begin{corollary}
		Let $(X,d)$ be a metric space. Then the upper Hausdorff metric topology on $CL(X)$ is the weakest topology such that each member of the family of set functionals $\{D_d(S,\cdot):S \in \mathcal{P}_0(X)\}$ is lower semi-continuous. 
	\end{corollary}

%
%
%
%
%
	
	\begin{remark}
		
			Note that each of the bornology $\mathcal{P}_0(X)$, $\mathcal{B}_d(X)$, and $\mathcal{F}(X)$ is shielded from closed sets and in each case we have $\tau_{\mathcal{S},d}^+ = \mathsf{G}_{\mathcal{S},d}^+$. 
			However, $\tau_{\mathcal{T}\mathcal{B}_d(X),d}^+ = \mathsf{G}_{\mathcal{T}\mathcal{B}_d(X),d}^+$ while $\mathcal{T}\mathcal{B}_d(X)$ need not be shielded from closed sets in general. For more on shields, see \cite{Bcas}.
			
		
	\end{remark}
	
We now give an example of a sequence $(A_n)$ of subsets of the Euclidean space $(\mathbb{R}^2, d_e)$ that appears geometrically not converging to $A \subseteq \mathbb{R}^2$ such that $(A_n)$ does not converge to $A$ with respect to $\tau_{\mathcal{S},d}^+$ but converges to $A$ with respect to $\mathsf{G}_{\mathcal{S},d}^+$.
	
	\begin{example}
		Let $(X,d) = (\mathbb{R}^2, d_e)$ and $\mathcal{B} = \{S \subseteq S_0 \cup F: F \in \mathcal{F}(X)\}$,  where $S_0 = \{(x,y) : x \geq 1, y = 1\}$. It is easy to verify that $\mathcal{B}$ is a bornology on $X$. Define $A_n = \{(x,y) : x \geq 0, y = n\}$ for each $n \in \mathbb{N}$ and $A = \{(x,y) : x = 0, y \geq 0\}$. Then $(A_n)$ is a sequence of closed sets. We claim that $(A_n)$ converges to $A$ in $\mathsf{G}_{\mathcal{B},d}^+$. Let  $S_1,\ldots,S_k \in \mathcal{B}$ and $\epsilon_1,\ldots,\epsilon_k > 0$ such that $A \in \mathscr{A}_d^+(S_1,\ldots,S_k; \epsilon_1,\ldots, \epsilon_k)$. Take $n_0 = \max\{y : (x,y) \in S_i, i = 1,\ldots,k\} + 2 \max \{\epsilon_i : i = 1,\ldots,k\}$. Observe that, for $n \geq n_0$, $A_n \in \mathscr{A}_d^+(S_1,\ldots,S_k;\epsilon_1,\ldots,\epsilon_k)$.  Thus, we have established $\mathsf{G}_{\mathcal{B},d}^+$ convergence.
	
		To see $(A_n)$ does not converge to $A$ in $\tau_{\mathcal{B},d}^+$. Take $S = \mathbb{N} \times \{1\}$ and let $ f,g \in \mathcal{Z}^+$ defined by $f((m,1)) = \frac{m}{8}$ and $g((m,1)) = \frac{m}{4}$ for each $m \in \mathbb{N}$. Observe that $\inf\{g((m,1)) - f((m,1)) : m \in \mathbb{N}\} = \frac{1}{8}$. Also $A \cap B_d(S,g) = \emptyset$.
		 Since for any $n \in \mathbb{N}$, $(16n,n) \in A_n \cap B_d((16n,1),\frac{16n}{8})$, $A_n \cap B_d(S, f) \neq \emptyset$ for every $n \in \mathbb{N}$. Thus, $(A_n)$ does not converge to $A$ in $\tau_{\mathcal{B},d}^+$. \qed
	\end{example}

%
%
%

	\section{Bornological Convergence}
	In \cite{Idealtopo}, the authors characterized the coincidence of $\mathcal{S}^+$-convergence and $\tau_{\mathcal{S},d}^+$-convergence for an arbitrary bornology $\mathcal{S}$. They used the notion of \textit{$\mathcal{S}$-convergent to infinity} to characterize the aforesaid coincidence. In this section, we first provide a new approach to determine $\mathcal{S}^+$-convergence which is similar to the one given for the $\tau_{\mathcal{S},d}^+$-convergence in Theorem \ref{Miss theorem for S topology}. Then we give necessary and sufficient conditions on the structure of the bornology $\mathcal{S}$ for the coincidence of $\mathcal{S}^+$-convergence and $\tau_{\mathcal{S},d}^+$-convergence on $CL(X)$.



	
	\begin{theorem}[Miss type characterization for bornological convergence]\label{upperbornologicalrepresentation}
		Let $(X,d)$ be a metric space and let $\mathcal{S}$ be a bornology on $X$. Suppose $(A_\lambda)$ is a net in $\mathcal{P}_0(X)$ $(CL(X))$ and $A$ is a nonempty (closed) subset of $X$. Then the following statements are equivalent. 
		\begin{enumerate} [(i)]
			\item  $(A_\lambda)$ $\mathcal{S}^+$-converges to $A$;
			
			\item for $S \in \mathcal{S}$ and $\epsilon > 0$, whenever $A$ misses $\epsilon$-enlargement of $S$, then $(A_\lambda)$ misses $S$ eventually.  
		\end{enumerate}
	\end{theorem} 
	\begin{proof}
		$(i) \Rightarrow (ii)$. Let $S \in \mathcal{S}$ and $\epsilon > 0$  such that $A \cap B_d(S, \epsilon) = \emptyset$. By the hypothesis, there is a $\lambda_0 \in \Lambda$ such that for all $\lambda \geq \lambda_0$, we have $A_\lambda \cap S \subseteq B_d(A, \epsilon)$. Consequently, by the choice of $S$ and $\epsilon$, $A_\lambda \cap S = \emptyset$ for all $\lambda \geq \lambda_0$.
		
		$(ii)\Rightarrow(i)$. Let $S \in \mathcal{S}$ and $\epsilon > 0$. If $A \cap B_d(S, \epsilon) = \emptyset$, then there is a $\lambda_1 \in \Lambda$ such that $A_\lambda \cap S = \emptyset$ for all $\lambda \geq \lambda_1$. Consequently,  $A_\lambda \cap S \subseteq B_d(A, \epsilon)$ for all $\lambda \geq \lambda_1$. Otherwise, let $B = \{x \in S : x \in B_d(A, \epsilon)\}$ and $S' = S \setminus B$. When $B = S$, then $A_{\lambda} \cap S \subseteq B_d(A, \epsilon)$ for all $\lambda$. If $B\neq S$, then $S' \neq \emptyset$.  Moreover, $A \cap B_d(S', \epsilon) = \emptyset$. Thus, by the hypothesis there is a $\lambda_2 \in \Lambda$ such that $A_\lambda \cap S' = \emptyset$ for all $\lambda \geq \lambda_2$. Also for each $\lambda \in \Lambda$,  $A_\lambda \cap B \subseteq B_d(A, \epsilon)$. Consequently, $A_\lambda \cap S \subseteq B_d(A, \epsilon)$ for all $\lambda \geq \lambda_2$. Hence $(A_\lambda)_{\lambda \in \Lambda}$ $\mathcal{S}^+$-converges to $A$.     
	\end{proof}


We would like to mention that the equivalence of Theorem \ref{upperbornologicalrepresentation} may not hold if we do not assume $\mathcal{S}$ to be a bornology  on $X$. This is shown by the following example.

	\begin{example}\label{counterexample_bornollogical}
		Suppose $X = \mathbb{R}^2$ and $d =  d_e$, the Euclidean metric. Let $\mathcal{S} = \{B_d((0,0), 2m) :  m \in \mathbb{N}\}$. Then $\mathcal{S}$ is a cover of $X$. Observe that $\downarrow\mathcal{S} = \{B \subseteq S: 	S \in \mathcal{S}\} = \mathcal{B}_d(X)$. So by Corollary $3.3$ of \cite{BCL}, we have $\mathcal{S}^+ = \tau_{AW_d}^+$. Let $A = \{(x,y): y = \frac{1}{x}, x > 0\}$, and $A_n = \{(x,y) : y = \frac{1}{x} + \frac{1}{n}, x < 0\}$ for each $n \in \mathbb{N}$. Note that $A \cap B_d((0,0), 2m) \neq \emptyset$ for any $m \geq 1$. Thus, the sequence $(A_n)$ and $A$ satisfy $(ii)$ of Theorem \ref{upperbornologicalrepresentation}. On the other hand, for any $0 < \epsilon < 1$ and $n \in \mathbb{N}$, $\emptyset \neq A_n \cap B_d((0,0), 2) \nsubseteq B_d(A, \epsilon)$. Therefore, the $\tau_{AW_d}^+$-convergence of sequence $(A_n)$ to $A$ fails.\qed 
		\end{example}
	
	
Recall that for a bornology $\mathcal{S}$ on a metric space $(X,d)$, the upper $\mathcal{S}$-proximal topology $\mu_{\mathcal{S}}^{++}$ on $\mathcal{P}_0(X)$ is generated by all sets of the form: $(S^c)^{++} = \{A \in \mathcal{P}_0(X):  S \cap B_d(A, \epsilon)= \emptyset~ \text{for some}~ \epsilon > 0  \}$, where $S \in \mathcal{S}$ (\cite{BCL, di2003bombay}).	The relation between $\mathcal{S}^+$-convergence and $\mu_{\mathcal{S}}^{++}$-convergence is given in Proposition $15$ of \cite{Idealtopo}. We reproduce its proof using Theorem \ref{upperbornologicalrepresentation}.
	
	\begin{proposition}
		Let $(X,d)$ be a metric space and let $\mathcal{S}$ be a bornology on $X$. Then the following statements are equivalent:
		\begin{enumerate}[(i)]
			\item for each $S \in \mathcal{S}$ there is an $\epsilon > 0$ such that $B_d(S,\epsilon) \in \mathcal{S}$, that is, $\mathcal{S}$ is stable under small enlargements;
			
			\item $\mathcal{S}^+$-convergence is compatible with $\mu_{\mathcal{S}}^{++}$-convergence on $\mathcal{P}_0(X)$;
			
			\item $\mathcal{S}^+$-convergence is topological on $\mathcal{P}_0(X)$. 
		\end{enumerate}
		
	\end{proposition}
	
	\begin{proof}
		$(i) \Rightarrow (ii)$. Suppose $(A_\lambda)$ is a net in $\mathcal{P}_0(X)$ that $\mathcal{S}^+$-converges to a nonempty set $A$. If $\overline{A} = X$, then there is nothing to show. Otherwise, let $S \in \mathcal{S}$ be such that $S \cap B_d(A, \epsilon) = \emptyset$, equivalently $A \cap B_d(S, \epsilon) = \emptyset$ for some $\epsilon > 0$. By assumption, we can find $0 < \delta < \frac{\epsilon}{2}$ such that $B_d(S, \delta) \in \mathcal{S}$. Then, $A \cap B_d(B_d(S,\delta),{\frac{\epsilon}{2}}) =  \emptyset$. Applying Proposition \ref{upperbornologicalrepresentation}, it can be concluded that, $A_\lambda \cap B_d(S,\delta) = \emptyset$ eventually. Thus, the net $(A_\lambda)_{\lambda \in \Lambda}$ $\mu_{\mathcal{S}}^{++}$-converges to $A$.
		 
		Conversely, suppose $(A_\lambda)$ $\mu_{\mathcal{S}}^{++}$-converges to $A$. Let $S \in \mathcal{S}$ and $\epsilon >0$ be such that $A \cap B_d(S,\epsilon) = \emptyset$. Then $A_\lambda \in (S^c)^{++}$ eventually. Thus, $A_\lambda \cap S = \emptyset$ eventually. Hence by Theorem \ref{upperbornologicalrepresentation}, $(A_\lambda)$ $\mathcal{S}^+$-converges to $A$.
		
		The implications $(ii) \Rightarrow (iii)$ is immediate and for $(iii) \Rightarrow (i)$ see Theorem $3.7$ of \cite{BCL}.   
	\end{proof}

	In \cite{Idealtopo}, the authors gave an example (Example 20) showing the condition that a family $\mathcal{S}$ being stable under enlargements by positive constants may not be sufficient to enforce the equivalence of $\tau_{\mathcal{S},d}^+$ and  $\mathcal{S}^+$-convergences. So we may need a stronger stability condition on $\mathcal{S}$ to achieve the required equivalence.  We now give the main result of this section that divulges this stronger stability condition on $\mathcal{S}$.
	
		\begin{theorem}\label{equivalence S topology and bornological}
		Let $(X,d)$ be a metric space and let $\mathcal{S}$ be a bornology on $X$. Then the following statements are equivalent:
		
		\begin{enumerate}[(i)]
				\item $\tau_{\mathcal{S},d}^+$-convergence $ = \mathcal{S}^+$-convergence on $CL(X)$;
			\item for each $S \in \mathcal{S}$ and $f, g \in \mathcal{Z}^+$ with $\inf\{g(x) - f(x) : x \in S\} > 0$, either $B_d(S,f) \in \mathcal{S}$ or $B_d(S,g) = X$.
				\end{enumerate} 
			\end{theorem}
	\begin{proof}
		$(i) \Rightarrow (ii)$. Suppose there exist $S_0 \in \mathcal{S}$ and $f,g \in \mathcal{Z}^+$ with $\inf_{x  \in S_0}(g(x) -f(x)) > 0$ for which $B_d(S_0, g) \neq X$ as well as $B_d(S_0, f) \notin \mathcal{S}$. Define $$ \Omega = \{S \in \mathcal{S}: \text{ there is an }\epsilon > 0 \text{ such that } B_d(S, \epsilon) \subseteq B_d(S_0, g) \}.$$ Note that $\Omega$ is nonempty as $S_0 \in \Omega$. By assumption, for each $S \in \Omega$ there is an $x_S \in B_d(S_0,f) \setminus S$. Put $A = X \setminus B_d(S_0, g)$ and for $S \in \Omega$, let $A_S = A \cup \{x_S\}$. Direct $\Omega$ by set inclusion. Then $(A_S)_{S \in \Omega}$ is a net in $CL(X)$. Since $A_S\cap B_d(S_0, f) \neq \emptyset$ for each $S \in \mathcal{S}$ and $A\cap B_d(S_0, g) =\emptyset$, by Theorem \ref{Miss theorem for S topology}, the $\tau_{\mathcal{S},d}^{+}$-convergence of $(A_S)$ to $A$ fails. To see $(A_S)$ $\mathcal{S}^+$-converges to $A$, let $S_1 \in \mathcal{S}$ and $\epsilon_1 > 0$ such that $A \cap B_d(S_1, \epsilon_1) = \emptyset$. So $S_1 \in \Omega$. Then for any $S \in \Omega$ with $S \supseteq S_1$, we have $A_S \cap S_1 = \emptyset$.  Therefore, by Theorem \ref{upperbornologicalrepresentation}, the net $(A_S)$ is $\mathcal{S}^+$-convergent to $A \in CL(X)$. This is a contradiction.
		
		$(ii) \Rightarrow (i)$. Suppose $(A_\lambda)$ is a net $\mathcal{S}^+$-converging to $A$ in $CL(X)$. If $A = X$, then there is nothing to show. Otherwise, let $S \in \mathcal{S}$, and $f, g \in \mathcal{Z}^+$ with $\inf_{x \in S}(g(x) - f(x)) = r > 0$ such that $A \cap B_d(S,g) = \emptyset$. By the hypothesis, we get $S' = B_d(S,f) \in \mathcal{S}$. Since $A \cap B_d(S', r) = \emptyset$ and the net $(A_S)$ is $\mathcal{S}^+$-convergent to $A$, by Theorem \ref{upperbornologicalrepresentation}, we obtain $A_{\lambda} \cap S' = \emptyset$ eventually. Thus, by Theorem \ref{Miss theorem for S topology}, the net $(A_{\lambda})$ is $\tau_{\mathcal{S},d}^+$-convergent to $A$. 
		\end{proof}

\begin{corollary}
Let $(X,d)$ be a metric space and let $\mathcal{S}$ be a bornology on $X$. If $B_d(S,f) \in \mathcal{S}$ for each $S \in \mathcal{S}$ and $f \in \mathcal{Z}^+$, then $\tau_{\mathcal{S},d}^+$-convergence $ = \mathcal{S}^+$-convergence on $CL(X)$.
\end{corollary}

If $d$ is an unbounded metric on $X$ and $\tau_{\mathcal{S},d}^+ = \mathcal{S}^+$, then by Theorem \ref{equivalence S topology and bornological}, each ball in $(X,d)$ is in $\mathcal{S}$. So whenever $\tau_{\mathcal{S},d}^+ = \mathcal{S}^+$ for an unbounded metric space $(X,d)$, then $\mathcal{B}_d(X) \subseteq \mathcal{S}$.  Moreover, in Proposition $19$ of \cite{Idealtopo}, it has been shown that whenever $(X,d)$ is unbounded and $\tau_{\mathcal{S},d}^+ = \mathcal{S}^+$, then $\mathcal{S}$ is stable under enlargements. The proof of this proposition uses the notion of $\mathcal{S}$-convergence to infinity. We next give a new proof of Proposition $19$ of \cite{Idealtopo} in the vein of Theorem \ref{equivalence S topology and bornological}.   

\begin{corollary}
	Let $(X,d)$ be an unbounded metric space and let $\mathcal{S}$ be a bornology on $X$. If $\tau_{\mathcal{S},d}^+ = \mathcal{S}^+$ on $CL(X)$, then $\mathcal{S}$ is stable under enlargements.  
\end{corollary}
\begin{proof}
	Let $S_0 \in \mathcal{S}$ and $r > 0$. If $B_d(S_0, r+ \epsilon) \neq X$ for some $\epsilon > 0$, then by Theorem \ref{equivalence S topology and bornological}, $B_d(S_0, r) \in \mathcal{S}$. Otherwise, $B_d(S_0, r+ \epsilon)= X$ for each $\epsilon  > 0$. Let $x_0 \in S_0$, and $S' = S_0 \setminus B_d(x_0, 3r)$. Since $(X,d)$ is unbounded, $S' \neq \emptyset$. Further, note that $B_d(S',\frac{5r}{2}) \neq X$ as $d(x_0, S') \geq 3r$. Consequently, by Theorem \ref{equivalence S topology and bornological}, we have $B_d(S', 2r) \in \mathcal{S}$. Also $B_d(x_0, 5r) \in \mathcal{S}$. Then $B_d(S_0, 2r) \in \mathcal{S}$ as $B_d(S_0, 2r) \subseteq B_d(S', 2r) \cup B_d(x_0, 5r)$. And, by our assumption, $X = B_d(S_0, 2r)$. Hence $\mathcal{S} = \mathcal{P}_0(X)$.
\end{proof}

\begin{proposition}\label{suffbornlogicalTsd}
	Let $(X,d)$ be a metric space and let $\mathcal{S}$ be bornology on $X$. Suppose for any $S \in \mathcal{S}$ and $f\in \mathcal{Z}^+$ either $B_d(S,f) \in \mathcal{S}$ or $\overline{B_d(S,f)} = X$. Then $\tau_{\mathcal{S},d}^+ = \mathcal{S}^+$ on $CL(X)$.  
\end{proposition}
\begin{proof}
	Suppose $(A_\lambda)$ is a net which $\mathcal{S}^+$-converges to $A$ in $CL(X)$. If $A = X$, then there is nothing to show. Otherwise, let $S \in \mathcal{S}$ and $f, g \in \mathcal{Z}^+$ with $\inf_{x \in S}(g(x) - f(x)) > 0$ such that $A \cap B_d(S,g) = \emptyset$. Then by the hypothesis, we get $B_d(S,f) \in \mathcal{S}$. Therefore, applying Theorem \ref{upperbornologicalrepresentation}, we get $A_\lambda \cap B_d(S,f) = \emptyset$ eventually. Thus, the coincidence $\tau_{\mathcal{S},d}^+ = \mathcal{S}^+$ follows from Theorem \ref{Miss theorem for S topology}. 
\end{proof}
The following example shows that the condition in Proposition \ref{suffbornlogicalTsd} need not be necessary for the coincidence of the convergences $\tau_{\mathcal{S},d}^+$ and $\mathcal{S}^+$. 

\begin{example}
	Let $X = [0,1] \cup E$ where $E = \cup_{p \in P}E_{p}$, $E_p = \{p^n: n \in \mathbb{N}\}$, and $P = \{p: p \text{ is prime}\}$. Suppose $d:X \times X \rightarrow \mathbb{R}$ is defined as $$d(x,y) = \begin{cases}
		2 & \text{ if } x,y \in [0,1] \text{ or } x \in [0,1], y \in E, x \neq y\\
		1+ \sum_{r = 1}^{ q}\frac{1}{2^r} & \text{ if }x \in E_p, y \in E_q, p \leq q, x \neq y\\
		0        &   \text{ if } x =y
	\end{cases}$$
	where $x,y \in X$. It is routine to verify that $d$ forms a metric on $X$. Define the collection $\mathcal{A} = \{E_p \cup F: p \in P, F \in \mathcal{F}(X)\}$. Then $\mathcal{A}$ forms a cover of $X$. By Proposition $11.2$ of \cite{beer2023bornologies}, $\mathcal{A}$ generates a bornology, say $\mathcal{B}(\mathcal{A})$ on $X$.
	
	First, we show that for $S \in \mathcal{B}(\mathcal{A})$ and $f, g \in \mathcal{Z}^+$ with $\inf_{x \in S}(g(x)-f(x)) > 0$, either $B_d(S,f) \in \mathcal{B}(\mathcal{A})$ or $B_d(S,g) = X$. Observe that any $S \in \mathcal{B}(\mathcal{A})$ is a subset of $\cup_{i =1}^n E_{p_i}\cup F$ for $p_i \in P$, $1 \leq i \leq n$ and $F \in \mathcal{F}([0,1])$. For $f \in \mathcal{Z}^+$ there are two possibilities: $\sup_{x \in S}f(x) < 2$ or $\sup_{x \in S}f(x) \geq 2$. When $\sup_{x \in S}f(x) < 2$, then we can find $p_0 \in P$ such that $p_0 > \max_{1 \leq i \leq n} p_i$ and $\sup_{x \in S}f(x) < 1+ \sum_{r = 1}^{ p}\frac{1}{2^r}~ \forall~ p \geq p_0$.  Therefore, $B_d(S,f) \subseteq \cup_{p \leq p_0}E_p \cup F \in \mathcal{B}(\mathcal{A})$. When $\sup_{x \in S}f(x) \geq 2$, then $g(x) > 2$ for some $x \in S$. So $B_d(S,g) = X$. Hence $\tau_{\mathcal{B}(\mathcal{A}), d}^+ = \mathcal{B}(\mathcal{A})^+$ follows by Theorem \ref{equivalence S topology and bornological}.
	
	Next, we show that the necessary condition in Proposition \ref{suffbornlogicalTsd} fails. Consider $E_2$ and $f \in \mathcal{Z}^+$ defined by $f(x = 2^n) = 2 - \frac{1}{n}$ for $ n \in \mathbb{N}$. We claim that $ E \subseteq B_d(E_2, f)$ and $\overline{B_d(E_2, f)} \neq X$. Let $y \in E_q$ for some $ q \in P$. Then there exists $n_0 \in \mathbb{N}$ such that $1+ \sum_{r = 1}^{ q}\frac{1}{2^r} < 2 - \frac{1}{n_0}$. Consequently, $y \in B_d(2^{n_0}, 2 - \frac{1}{n_0}) \subseteq B_d(E_2, f)$. So $ E \subseteq B_d(E_2, f)$. Therefore $B_d(E_2, f) \notin \mathcal{B}(\mathcal{A})$. Finally, note that for any $z \in [0,1]$, we have $ B_d(z, 1) \cap B_d(E_2, f) = \emptyset$. Hence $\overline{B_d(E_2, f)} \neq X$.\qed  
	
\end{example}

   
\begin{theorem}
	Let $(X,d)$ be an almost convex metric space and let $\mathcal{S}$ be a bornology on $X$. Then the following statements are equivalent:
	\begin{enumerate}[(i)]
		\item for $S \in \mathcal{S}$ and $f \in \mathcal{Z}^+$ either $B_d(S,f) \in \mathcal{S}$ or $\overline{B_d(S,f)} = X$;
		\item $\tau_{\mathcal{S},d}^+ = \mathcal{S}^+$ on $CL(X)$. 
	\end{enumerate}
\end{theorem}
\begin{proof}
	It is enough to prove $(ii) \Rightarrow (i)$.
	Suppose $(i)$ fails. So there exist $S \in \mathcal{S}$ and $f \in \mathcal{Z}^+$ such that neither $B_d(S,f) \in \mathcal{S}$ nor $\overline{B_d(S,f)} = X$. Let $p \in X \setminus \overline{B_d(S,f)}$. Put $2r = d(p, \overline{B_d(S,f)})$. Consider $g \in \mathcal{Z}^+$ with $g(x) = f(x) + r  ~\forall~ x \in X$. Since $(X,d)$ is an almost convex metric space, $B_d(B_d(x, f(x)), r) = B_d(x, f(x) + r)$ for $x \in S$. Therefore, $B_d(B_d(S,f),r) = B_d(S,g)$. Consequently, $ p  \notin B_d(S,g)$. Thus $B_d(S,g) \neq X$. Which contradicts Theorem \ref {equivalence S topology and bornological}.    
\end{proof}

	If $\tau_{\mathcal{S},d}^+ = \mathcal{S}^+$, then we also have $\mathcal{S}^+ = \mathsf{G}_{\mathcal{S},d}^+$.  However, in general ($\mathcal{S}^+ = \mathsf{G}_{\mathcal{S},d}^+)  \nRightarrow   (\tau_{\mathcal{S},d}^+ = \mathcal{S}^+)$. This can  be seen by the following example. 
	\begin{example}
		Let $X = \mathbb{R}^2$ and $d = d_e$. Let $\mathcal{S}$ be the bornology defined as $\mathcal{S} = \{S \subseteq \mathbb{R} \times [-n,n]: n \in \mathbb{N}\}$. It follows from Example $20$ of \cite{Idealtopo} that $\tau_{\mathcal{S},d}^+ \neq \mathcal{S}^+$ on $CL(X)$.  Since for any $S \in \mathcal{S}$ and $r > 0$, $\overline{B_d}(S,r) = \{x \in X: d(x, S) \leq r\} \in \mathcal{S}$, by Theorem $3$ of \cite{gapexcess}, $\mathsf{G}_{\mathcal{S},d}^+ = \mathcal{S}^+$ on $CL(X)$. 
%
	\end{example}

%
%

	Now as applications of Theorem \ref{equivalence S topology and bornological}, we deduce the coincidence $\tau_{\mathcal{S},d}^+ = \mathcal{S}^+$ for some special bornologies. In particular, we consider the relation of $\mathcal{S}^+$-convergence with $\tau_{\mathcal{S},d}^+$-convergence for $\mathcal{S} = \mathcal{F}(X),~ \mathcal{K}(X)$, $\mathcal{T}\mathcal{B}_d(X)$, $\mathcal{B}_d(X)$, and $\mathcal{P}_0(X)$. Note that in each of these special cases, we get different $\mathcal{S}^+$-convergences. However, when $\mathcal{S} = \mathcal{F}(X)$, $\mathcal{K}(X)$, and $\mathcal{T}\mathcal{B}_d(X)$, we have $\tau_{\mathcal{S},d}^+  = \tau_{W_{d}}^+$. So in these three cases, we actually deduce the equivalence $\tau_{W_d}^+ = \mathcal{S}^+$ using Theorem \ref{equivalence S topology and bornological}. We would like to mention here that in Theorem 1 of \cite{gapexcess}, Beer and Levi have proved a much more general result where they have compared the $\tau_{W_{d}}^+$-convergence to $\mathcal{S}^+$-convergence for an arbitrary bornology $\mathcal{S}$.

	
	\begin{corollary} $($\cite{ToCCoS}$)$
		Let $(X,d)$ be a metric space. Then the following statements hold. 
		\begin{enumerate}[(i)]
			
			\item $\tau_{AW_d}^+$-convergence $= \mathcal{B}_d^+(X)$-convergence on $CL(X)$;
			
			\item $\tau_{H_{d}}^+$-convergence = $\mathcal{P}_0^+(X)$-convergence on $CL(X)$.
		\end{enumerate} 
	\end{corollary}
	
	\begin{proof}
		$(i)$. Let $B \in \mathcal{B}_d(X)$, and $f \in \mathcal{Z}^+$. Then by Proposition \ref{Bounded set}, either $B_d(B,f) \in \mathcal{B}_d(X)$ or $B_d(B,f) = X$. Therefore the result follows from Theorem \ref{equivalence S topology and bornological}. 
		
		$(ii)$. It is immediate from Theorem \ref{equivalence S topology and bornological}.	\end{proof}
	
	 In order to examine other cases, we first prove a following general result which can be inferred from Theorem $8$ of \cite{Idealtopo}. For the readers convenience and sake of completeness we record its proof below. 
	

	\begin{lemma}\label{upper S topology diff bornology}
		Let $(X,d)$ be a metric space and let $\mathcal{S}$ be a bornology on $X$. If $\mathcal{S} \subseteq \mathcal{T}\mathcal{B}_{d}(X)$, then $\tau_{\mathcal{S},d}^+ = \tau_{W_{d}}^+$ on $CL(X)$. 
	\end{lemma}
	
	\begin{proof}
		Suppose $(A_\lambda)$ is a net in $CL(X)$ that $\tau_{W_{d}}^+$-converges to $A \in CL(X)$. We show that the net $(A_\lambda)$ is $\tau_{\mathcal{S},d}^+$-convergent to $A$. Take $S \in \mathcal{S}$ and $\epsilon > 0$. Choose $F \in \mathcal{F}(X)$ such that $S \subseteq B_d(F, \frac{\epsilon}{4})$. By assumption, there is a $\lambda_0$ such that for all $\lambda \geq \lambda_0$ and $y \in F$, we have $d(y, A) - d(y, A_\lambda) < \frac{\epsilon}{4}$. Then for any $x \in S$ and $\lambda \geq \lambda_0$ we have $d(x, A) < d(x, A_\lambda) + \epsilon$.  
	\end{proof}
		\begin{proposition}\label{closedballresult}
		Let $(X,d)$ be a metric space and let $\mathcal{S}$ be a bornology on $X$ such that $\mathcal{S} \subseteq \mathcal{T}\mathcal{B}_d(X)$. Then the following statements are equivalent: 
		\begin{enumerate}[(i)]
			\item each proper closed ball is in $\mathcal{S}$;
			\item for $S \in \mathcal{S}$ and $f,g \in \mathcal{Z}^+$ with $\inf_{x \in S}(g(x) -f(x)) > 0$ whenever $B_d(S,g) \neq X$, we have $B_d(S,f) \in \mathcal{S}$. 
		\end{enumerate}
	\end{proposition}
	
	\begin{proof}$(i) \Rightarrow (ii)$. Let $S \in \mathcal{S}$ and $f,g \in \mathcal{Z}^+$ with $\inf_{x \in S}(g(x)- f(x)) > 0$ such that $B_d(S,g) \neq X$. By Lemma \ref{upper S topology diff bornology}, we have $\tau_{\mathcal{S},d}^+ = \tau_{W_{d}}^+ \subseteq \mathsf{G}_{\mathcal{S},d}^+$. So $\tau_{\mathcal{S},d}^+ = \mathsf{G}_{\mathcal{S},d}^+$. Then by Theorem \ref{gap and Stopology equivaalence}, we have $B_d(S,f)$ is strictly $(\mathcal{S}-d)$ included in $B_d(S,g)$. Let $S_1,\ldots, S_n \in \mathcal{S}$ and $0 < \alpha_i < \epsilon_i$ for $1\leq i\leq n$ such that $$B_d(S,f) \subseteq \cup_{i = 1}^n B_d(S_i, \alpha_i)\subseteq \cup_{i = 1}^n B_d(S_i, \epsilon_i) \subseteq B_d(S,g).$$  Choose $0 < r_i < \epsilon_i - \alpha_i$ for $ 1 \leq i \leq n$. Since each $S_i$ is totally bounded, there exist $x_{i_1},.., x_{i_k} \in S_i$ such that $S_i \subseteq \cup_{j=1}^{k}B_d(x_{i_j},r_i)$. Consequently, $$B_d(S_i, \alpha_i) \subseteq \cup_{j = 1}^{k}B_d(x_{i_j}, \alpha_i + r_i) \subseteq B_d(S_i, \epsilon_i).$$ By the hypothesis, $\cup_{j = 1}^{k}B_d(x_{i_j}, \alpha_i + r_i) \in \mathcal{S}$. Thus, $B_d(S_i, \alpha_i) \in \mathcal{S}$ for each $i=1,\ldots,n$. Hence $B_d(S,f) \in \mathcal{S}$.
		   
		$(ii)\Rightarrow (i)$. Let $x \in X$ and $\epsilon > 0$ be such that $\overline{B_d}(x, \epsilon) \neq X$. Then there exists a $y \in X$ such that $d(x, y) > \epsilon$. Put $d(x,y) - \epsilon = 3r$. So $B_d(x, \epsilon + 3r) \neq X$. By the hypothesis, we have $B_d(x, \epsilon +2r) \in \mathcal{S}$. Hence $\overline{B_d}(x, \epsilon) \in \mathcal{S}$.   
	\end{proof}
	
	\begin{corollary}\label{compactsetbornology} $($\cite{ToCCoS}$)$
		Let $(X,d)$ be a metric space. Then 
	  $\tau_{W_{d}}^+$-convergence = $\mathcal{K}(X)^+$-convergence if and only if $(X,d)$ have nice closed balls (that is, each proper closed ball is compact). 
	\end{corollary}

\begin{proof}
Since $\mathcal{K}(X)\subseteq \mathcal{T}\mathcal{B}_d(X)$, the result follows from Lemma \ref{upper S topology diff bornology}, Proposition \ref{closedballresult} and Theorem \ref{equivalence S topology and bornological}.
 \end{proof}

By Corollary $3.3$ of \cite{BCL}, we have $s(X)^+ = \mathcal{F}(X)^+$, where $s(X)$ = collection of all singletons in $X$. 

	\begin{corollary} $($Proposition $22$, \cite{Idealtopo}$)$
	Let $(X,d)$ be a metric space. The following statements are equivalent: 
	\begin{enumerate}[(i)]
		\item $\mathcal{F}(X)^+$-convergence ensures $\tau_{W_{d}}^+$-convergence on $CL(X)$;
		\item either each closed ball in $(X,d)$ is finite or $(X,d)$ is bounded.
	\end{enumerate} 
\end{corollary}

Our next corollary also follows from Lemma \ref{upper S topology diff bornology}, Proposition \ref{closedballresult} and Theorem \ref{equivalence S topology and bornological}. However, we give a direct proof without using Proposition \ref{closedballresult}. The proof of $(ii)\Rightarrow (i)$ in the following corollary can be imitated to give an alternative proof of $(1) \Rightarrow (2)$ in Theorem 1 of \cite{gapexcess}.  	
	
	\begin{corollary}\label{totallyboundedbornology}
		Let $(X,d)$ be a metric space. Then the following statements are equivalent:
		\begin{enumerate}[(i)]
			\item $\mathcal{T}\mathcal{B}_{d}(X)^+$-convergence ensures $\tau_{W_{d}}^+$-convergence on $CL(X)$;
			\item each proper closed ball in $(X,d)$ is totally bounded.
		\end{enumerate}
	\end{corollary}
	\begin{proof}
			$(i) \Rightarrow (ii)$. Let $x \in X$ and $\epsilon > 0$ be such that $\overline{B_d}(x, \epsilon) \neq 	X$. Then there exists a $y \in X$ such that $d(x, y) > \epsilon$. Put $d(x,y) - \epsilon = 3r$. So $B_d(x, \epsilon + 3r) \neq X$. Therefore by Theorem \ref{equivalence S topology and bornological}, we have, $B_d(x, \epsilon +2r) \in \mathcal{T}\mathcal{B}_d(X)$. Hence $\overline{B_d}(x, \epsilon) \in \mathcal{T}\mathcal{B}_d(X)$.    
		
		$(ii) \Rightarrow (i)$.  Let $(A_{\lambda})$ be a net that $\mathcal{T}\mathcal{B}_d(X)^+$-converges to $A$ in $CL(X)$. Suppose $x \in X$ and $0< \alpha < \epsilon $ such that $A \cap B_d(x, \epsilon) = \emptyset$. Set $\overline{B_d}(x, \alpha) = S$. Then by the hypothesis, $S \in \mathcal{T}\mathcal{B}_d(X)$. Clearly, $A \cap B_d(S, \frac{\epsilon - \alpha}{2}) = \emptyset$. So by $\mathcal{T}\mathcal{B}_d(X)^+$-convergence of net $(A_{\lambda})$ to $A$ and using Theorem \ref{upperbornologicalrepresentation}, there exists $\lambda_0$ such that $A_{\lambda} \cap S = \emptyset ~ \forall \lambda \geq \lambda_0$, that is, $A_{\lambda} \cap B_d(x, \alpha) = \emptyset ~ \forall \lambda \geq \lambda_0$. Thus, the net $(A_{\lambda})$ is  $\tau_{W_{d}}^+$-convergent to $A$ in $CL(X)$. 
	\end{proof}

	\bibliographystyle{plain}
	\bibliography{reference_file}

\end{document}